\documentclass[a4paper,12pt,intlimits,oneside]{amsart}

\usepackage{enumerate}

\usepackage{latexsym,amssymb}

\newcommand{\comment}[1]{}

\newcommand{\bR}{{\mathbb R}}

\newcommand{\bZ}{{\mathbb Z}}

\newcommand{\bB}{{\mathbb B}}

\def\H{{\mathcal H}}

\def\F{{\mathcal F}}
\def\h{\mathfrak h}

\def\M{{\mathcal M}}
\def\S{{\mathcal S}}

\def\bmo{{\mathfrak {bmo}}}

\def\BMO{B\! M\! O}

\def\div{{\mbox{\small\rm  div}}\,}
\def\curl{{\mbox{\small\rm  curl}}\,}
\def\Msob{{\M_{ Sob}}}

\newcounter{rea}
\newcounter{rej}
\newcounter{res}
\newtheorem{thm}{Theorem}[section]

\newtheorem{lem}[thm]{Lemma}
\newtheorem{defn}[thm]{Definition}

\begin{document}

\title[]{Endpoint for the div-curl lemma \\
in Hardy spaces}

\author[A. Bonami]{Aline Bonami}
\address{MAPMO-UMR 6628,
D\'epartement de Math\'ematiques, Universit\'e d'Orleans,
45067 Orl\'eans Cedex 2, France}
\email{{\tt Aline.Bonami@univ-orleans.fr}}
\author[J. Feuto]{Justin Feuto}
\address{Laboratoire de Math\'ematiques Fondamentales, UFR Math\'ematiques et Informatique, Universit\'e de Cocody, 22 B.P 1194 Abidjan 22. C\^{o}te d'Ivoire}
\email{{\tt justfeuto@yahoo.fr}}
\author[S. Grellier]{Sandrine Grellier}
\address{MAPMO-UMR 6628,
D\'epartement de Math\'ematiques, Universit\'e d'Orleans,
45067 Orl\'eans Cedex 2, France}
\email{{\tt Sandrine.Grellier@univ-orleans.fr}}

\thanks{The authors are partially supported by the project ANR AHPI number ANR-07-BLAN-0247-01.}
\begin{abstract} We give a
$\div$-$\curl$ type lemma for the wedge product of closed
differential forms on $\bR^n$ when they have coefficients
respectively in a Hardy space and  $L^\infty$ or a space of $BMO$ type. In this last case,
the wedge product belongs to an appropriate Hardy-Orlicz space.
\end{abstract}
\maketitle

\section{Introduction}
The theory of compensated
compactness initiated and developed by L. Tartar
\cite{{Ta}} and F. Murat \cite{{M}} has been largely studied and extended to various
setting. The famous paper of Coifman, Lions, Meyer and Semmes ( \cite{CLMS}) gives an overview of this theory in the context of Hardy spaces in the Euclidean space $\bR^{n}$ $(n\geq
1)$. They prove in particular,
that, for $\frac{n}{n+1}<p,q<\infty$ such that
$\frac{1}{p}+\frac{1}{q}<1+\frac{1}{n}$,  when $F$ is a vector
field belonging to the Hardy space $\H^{p}(\bR^{n},\bR^{n})$ with
$\curl F=0$ and $G$ is a vector field belonging to
$\H^{q}(\bR^{n},\bR^{n})$ with $\div G=0$, then the scalar product
$F\cdot G$ can be given a meaning as a distribution of
$\H^{r}(\bR^{n})$ with
\begin{equation}
\left\|F\cdot G\right\|_{\H^{r}(\bR^{n})}\leq C\left\|F\right\|_{\H^{p}(\bR^{n},\bR^{n})}\left\|G\right\|_{\H^{q}(\bR^{n},\bR^{n})}
\end{equation}
where $\frac{1}{r}=\frac{1}{p}+\frac{1}{q}$.

The endpoint $\frac{n}{n+1}$ is related to cancellation properties
of Hardy spaces: bounded functions with compact support and zero
mean do not belong to $\H^{\frac n {n+1}}$ unless their moments of
order one are zero, a property that the scalar product $F.G$ does
not have in general.

We shall consider here the endpoint $q=\infty$. Let us first start
by some description of what is known. Auscher, Russ and
Tchamitchian remarked in \cite{ART} that, for $p=1$, one has,
under the same assumptions of being respectively curl free and
divergence free,
\begin{equation}
\left\|F\cdot G\right\|_{\H^{1}(\bR^{n})}\leq C\left\|F\right\|_{\H^{1}(\bR^{n},\bR^{n})}\left\|G\right\|_{L^{\infty}(\bR^{n},\bR^{n})}.
\end{equation}
In fact it is easy to see that the proof given in \cite{CLMS} is
also valid for $q=\infty$. They give in \cite{ART} another
proof, which has its own interest and has helped us in our
generalization to $\BMO$. Remark that the scalar product does not
make sense in general when $F$ is in $\H^{p}(\bR^{n},\bR^{n})$ for
$p<1$, so that one can only write a priori estimates, such as the
following one.

\begin{thm}\label{main01}
Let $\frac{n}{n+1}<p\leq1$.  If  $F\in \H^{p}(\bR^{n},\bR^{n})$ is
integrable and such that $\curl\,F=0$ and if $G\in
L^{\infty}(\bR^{n},\bR^{n})$ with $\div G=0$, then there exists a
constant $C$, independent of $F$ and $G$, such that
\begin{equation}\label{p-div-curl}
\left\| F\cdot G\right\|_{\H^{p}(\bR^{n})}\leq C
\left\|F\right\|_{\H^{p}(\bR^{n},\bR^{n})}\left\|G\right\|_{L^{\infty}(\bR^{n},\bR^{n})}.
\end{equation}
\end{thm}

This a priori estimate allows to give a meaning to $F\cdot G$ in
the distribution sense. There is no hope to give such a meaning
for general products $fg$, with $f\in \H^p(\bR^n)$ and $g\in
L^\infty$. It is proved in \cite{BF} that this is possible when
$g$ is in the inhomogeneous Lipschitz space
$\Lambda_\alpha(\bR^n)$, with $\alpha=n(\frac 1 p-1)$. Moreover,
one has
\begin{equation}
fg\in L^{1}(\bR^{n})+\H^{p}(\bR^{n})
\end{equation}
for $f$ in $\H^{p}$ (with $p<1$) and $g$ in
$\Lambda_{n\left(\frac{1}{p}-1\right)}$. So, cancellation
properties of the scalar product of curl free and divergence free
vector fields allow to get rid of the integrable part and to
weaken the assumptions from Lipschitz to bounded.

We will show in Section 4 that this generalizes to wedge products
of closed forms. Remark that end-point estimates would imply all
other ones by interpolation if we could interpolate between $\H^p$
spaces of closed forms. Indeed, for instance, the generalization
to closed forms allows to have \eqref{p-div-curl} when assumptions
on the two factors are exchanged: $F$ is bounded and $G$ is in $
\H^{p}(\bR^{n},\bR^{n})$. Unfortunately, one does not know whether
one can interpolate: while there is a bounded projection onto
closed forms in $\H^p$ for $p<\infty$, it is not the case for
$p=\infty$.

\bigskip

The core of this paper concerns $\div$-$\curl$ lemmas (and their
extensions for the wedge product of closed forms) when the
assumption to be bounded is weakened into an assumption of type
$\BMO$.  Products of functions in $\H^1$ and $\BMO$ have been
considered by Bonami, Iwaniec, Jones and Zinsmeister in
\cite{BIJZ}. Such products make sense as distributions, and can be
written as the sum of an integrable function and a function in a
weighted Hardy-Orlicz space. In order to have a $\div$-$\curl$ lemma in this context, we  make some restriction for one of the two factors. Recall that  $\bmo:=\bmo(\bR^{n})$ is the set of
locally integrable functions $b$ satisfying
\begin{equation}\label{bmo}
\sup_{|B|\leq 1}\left( \frac 1 {|B|}\int_B
|b-b_B|dx\right)<\infty\ \ \ \ \mbox{\rm and } \sup_{|B|\geq
1}\left( \frac 1 {|B|}\int_B |b|dx\right)<\infty
\end{equation}
with $B$ varying among all balls of $\bR^n$ and $|B|$ denoting the
measure of the ball $B$. The sum of the two finite quantities will
be denoted by $\left\|b\right\|_{\bmo}$. Then $\bmo$ is well-known to be  the dual space of the localized version of the Hardy space, which we note $\h^{1}(\bR^{n})$, see \cite{G}.
 To be more precise, for $f\in\H^1(\bR^n)$ and $g\in \bmo$, we define the product (in the distribution
 sense)
$f g$ as the distribution whose action on the Schwartz function
$\varphi\in\mathcal S(\bR^{n})$ is given by
\begin{equation}
\left\langle f g,\varphi\right\rangle:=\left\langle \varphi
g,f\right\rangle,
\end{equation}
where the second bracket stands for the duality bracket between
$\H^1$ and $\BMO$. It is then proved in
\cite{BIJZ} that
\begin{equation}\label{nocancel}
f g\in L^{1}(\bR^{n})+\H^{\Phi}_\omega(\bR^{n}).
\end{equation}
Here $\mathcal H^\Phi_\omega(\bR^n)$ is the weighted Hardy-Orlicz space related to the
Orlicz function
\begin{equation}\label{orl-log}
\Phi(t):=\frac t{\log (e+t)}
\end{equation}
and with weight $\omega(x):=(\log (e+|x|))^{-1}$.
This extends immediately to vector-valued functions. In the
next theorem, we prove that there is no $L^1$ term   in the
context of the $\div$-$\curl$ lemma.

\begin{thm}\label{bijz1}
Let $F\in \H^{1}(\bR^{n},\bR^{n})$ with ${\curl}F=0$ and $G\in
\bmo (\bR^{n},\bR^{n})$ with $\div G=0$. Then  there exists some
constant $C$, independent of $F$ and $G$, such that
\begin{equation}
\left\|F\cdot G\right\|_{\H^{\Phi}_\omega(\bR^{n})}\leq
C\left\|F\right\|_{\H^{1}(\bR^{n},\bR^{n})}\left\|G\right\|_{\bmo(\bR^{n},\bR^{n})}.
\end{equation}
\end{thm}

The theorem is also valid for $G\in \BMO$, but for local Hardy spaces $\h^1$ and $\h^\Phi_\omega$ instead of $\H^1$ and $\H^\Phi_\omega$. We do not know whether it is valid without this restriction on $F$ or $F\cdot G$. There is an $\H^p$ version of this theorem, for $p<1$, which we
give also. Note that $\div$-$\curl$ have been developed in the context of local Hardy spaces by Dafni \cite{D}.

These results can be compared to what can be said on products of
holomorphic functions, for which analogous estimates are elementary and have a weak converse, see  \cite{BG}.

\medskip

To simplify notations, we restricted to vector fields in the
introduction, but we shall write below these results in the context of
the wedge product of closed differential forms. Indeed, recall
that a divergence free vector field can be identified with an
$(n-1)$-closed differential form, while a $\curl$ free vector
field identifies with a $1$- closed form, their scalar product being given by the wedge product of these two forms.
The usual $\div$-$\curl$
lemma has been extended to wedge products of closed differential
forms by Lou and Mc Intosh \cite{LM1, LM2} when $\frac 1p+\frac
1q=1$, with both $p$ and $q$ finite. We will do it in general.

\medskip

Our paper is organized as follows. We recall basic results about
classical Hardy spaces in the second section. We define an
appropriate grand maximal function to characterize
$\H^{p}(\bR^{n})$, which has been introduced in \cite{ART}. In
Section 3, after recalling  some  basic facts about differential
forms, we give the analogous of the previous grand maximal
function characterization in this context. In Section 4 we give the whole range of
the $\div$-$\curl$ Lemma   for closed forms. Section 5 is devoted
to assumptions of type $\BMO$.

\medskip

Throughout this paper, $C$ denotes constants that are independent
of the functions involved, with values which may differ from line
to line. For two quantities $A$ and $B$, the notation $A\sim B$
means that there exist two positive  constants $C_{1}$ and $C_{2}$
such that $C_{1}A\leq B\leq C_{2}A$. If $E$ is a measurable subset
of $\bR^{n}$, then $\left|E\right|$ stands for its Lebesgue
measure.

\section{Some basic facts about classical Hardy spaces}

We fix $\varphi\in\S(\bR^{n})$  having integral $1$ and support in the unit
 ball $\mathbb B=\left\{x\in\bR^{n}:|x|<1\right\}$. For
  $f\in\mathcal S' (\bR^{n})$ and $x\in\bR^{n}$, we put
\begin{equation}
 \left(f\ast\varphi\right)(x):=\left\langle f,\varphi(x-\cdot)\right\rangle,
\end{equation}
 and define the maximal function $\mathcal M f=\mathcal M_{\varphi}f$ by
\begin{equation}\label{maximal}
\mathcal M f(x):=\sup_{t>0}\left|\left(f\ast\varphi_{t}\right)(x)\right|,
\end{equation}
where $\varphi_{t}(x)=t^{-n}\varphi\left(t^{-1}x\right)$.

For $p>0$,  a tempered distribution $f$ is said to belong to the
Hardy space $\H^{p}(\bR^{n})$ if
\begin{equation}\label{def}
\|f\|_{\H^{p}(\bR^{n})}:=\left
(\int_{\bR^n}\mathcal M_{\varphi}f(x)^{p}dx\right)^{\frac 1p}=\left\|\mathcal M_{\varphi}f\right\|_{L^{p}}
\end{equation}
 is finite.
It is well known that, up to equivalence of corresponding norms,
the space $\H^p(\bR^n)$  does not depend on the choice of the
function $\varphi$. So, in the sequel, we shall use the notation
$\M f$ instead of $\M_{\varphi}f$.

For $\frac{n}{n+1}<p\leq 1$, an $\H^{p}$-atom (related to the ball
$B$) is a bounded function $a$ supported in  $B$ and satisfying
the following conditions
\begin{equation}
\left\|a\right\|_{L^{\infty}}\leq\left|B\right|^{-\frac{1}{p}}\text{ and }\int_{\bR^{n}}a(x)dx=0.
\end{equation}
The atomic decomposition of $\H^{p}$ states that a temperate distribution $f$ belongs to $\H^{p}$
 if and only if there exist a sequence $(a_{j})$ of $\H^{p}$-atoms and a sequence $(\lambda_{j})$
  of scalars such that
\begin{equation}\label{atomic}
f=\sum^{\infty}_{j=1}\lambda_{j}a_{j}\quad\text{ and }
\quad\sum^{\infty}_{j=1}\left|\lambda_{j}\right|^{p}<\infty,
\end{equation}
where the first sum is assumed to converge in the sense of distributions. Moreover, $f$ is
 the limit of the partial sums in $\H^{p}$, and $\left\|f\right\|_{\H^{p}}$ is equivalent
 to the infimum, taken over all such decomposition of $f$, of the quantities
 $\left(\sum^{\infty}_{j=1}\left|\lambda_{j}\right|^{p}\right)^{\frac{1}{p}}$.

 We refer to \cite{St} for  background on Hardy spaces.

 For the purpose of our main results, we are going to define an appropriate
  grand maximal function, which induces on $\H^{p}$ a semi-norm equivalent to the previous one.

Let $q>n$. For $x\in\bR^{n}$, we denote by $\F^{q}_{x}$, the set
of all $\psi\in W^{1,q}(\bR^{n})$ supported in some ball $B(x,r)$
centered at $x$ with radius $r>0$ which satisfy
\begin{equation}
 \left\|\psi\right\|_{L^{q}(\bR^{n})}+r\left\|\nabla\psi\right\|_{L^{q}(\bR^{n})}
 \leq|B_{(x,r)}|^{-\frac{1}{q'}},
 \end{equation}
 where $\frac{1}{q}+\frac{1}{q'}=1$. Here $W^{1,q}(\bR^{n})$
 denotes the Sobolev space of functions in $L^q$ with derivatives
 in $L^q$. Since $q>n$, the Sobolev theorem guarantees that the
 test functions are bounded, which allows to give the following
 definition.

 For $f\in L^{1}_{loc}(\bR^{n})$, and $x\in\bR^{n}$, put
\begin{equation}
\mathcal M_{q} f(x):=\sup_{\psi\in
\F^{q}_{x}}|\int_{\bR^{n}}f\psi|.\label{grand-maximal}
\end{equation}

 The following lemma is classical, but we give its  proof for completeness.

 \begin{lem}\label{classical-result1}
 Let $f$ be a locally integrable function on $\bR^{n}$.
\begin{enumerate}
    \item [(i)] There exists a constant $C$ not depending on $f$, such that
    \begin{equation}
\mathcal Mf\leq C\mathcal M_{\infty}f,\label{point-wise1}
\end{equation}
\item [(ii)] For $\frac{n}{n+1}<p\leq1$ and $\frac 1q <\frac
{n+1}n-\frac 1p$
\begin{equation}
  \left\|\mathcal M_{q} f\right\|_{L^{p}(\bR^{n})}\sim \left\|f\right\|_{\H^{p}(\bR^{n})}.\label{norm-equi1}
\end{equation}
\end{enumerate}
\end{lem}

\begin{proof}
Let $f\in L^{1}_{loc}(\bR^{n})$.
 To prove (i), it is sufficient to see that, for
 $\varphi$  the test function used in the definition of Hardy space, there exists some constant $c$ such
 that,
 for all $x\in\bR^{n}$ and $t>0$, the function $\varphi_{x,t}(y):=c\varphi_{t}(x-y)$
belongs to $\F^{\infty}_{x}$.  One can choose
$c=\left(\left\|\varphi\right\|_{L^{\infty}(\bR^{n})}+
\left\|\nabla\varphi\right\|_{L^{\infty}}\right)^{-1}$.

Let us now prove (ii). It is sufficient to consider $q<\infty$ and
the inequality
\begin{equation}
 \left\|\mathcal M_{q} f\right\|_{L^{p}}\leq C \left\|f\right\|_{\H^{p}},\label{norm-control1}
 \end{equation} since
 \begin{equation}
 \mathcal M f\leq C\mathcal M_{\infty}f\leq C\mathcal M_{q}f.
 \end{equation}
 By sub-linearity of the maximal operator $\M_q$, it is sufficient
 to prove a uniform estimate for atoms,
 \begin{equation}
 \left\|\mathcal M_{q} a\right\|_{L^{p}}\leq C \label{norm-control2}
 \end{equation}
 for some uniform constant $C$. Indeed, once we have this, we conclude for $f=\sum \lambda_j a_j$
 that
$$\left\|\mathcal M_{q} f\right\|_{L^{p}}\leq \left(\sum |\lambda_j|^p\|\mathcal M_{q}
a\|_{L^{p}}^p\right)^{1/p}\leq C
\left(\sum|\lambda_j|^p\right)^{1/p}.$$

So let us prove \eqref{norm-control2}. Without loss of generality,
using invariance by translation and dilation, we may assume that
$a$ is a  function with zero mean, supported by the unit ball $\mathbb B$
centered at $0$, and bounded by $1$.   We prove that there
exists $\varepsilon>0$ depending on $q$ such that
$$|\mathcal M_{q} a(x)|\leq C(1+|x|)^{-n-1+\varepsilon}.$$
By assumption
on $\psi\in\F^q_x$, using H\"older's Inequality, we find that
$\|\psi\|_1\leq 1$. So $\mathcal M_{q} a$ is bounded by $1$, and
it is sufficient to prove that, for $\psi\in\F^q_x$,
$$|\int_{\mathbb{B}}\psi a| \leq C|x|^{-n-1+\varepsilon}$$ for $|x|\geq
2$. Moreover, in this range of $x$, we may restrict to functions
$\psi$ supported in $B(x,r)$ with $r>|x|/2$, so that $\|\nabla
\psi\|_q\leq C|x|^{-\frac n {q'}-1}$.

Since $a$ has mean zero, $$|\int_{\mathbb{B}}\psi
a|=|\int_{\mathbb{B}}(\psi-\psi_{\mathbb{B}}) a|\leq C\|\nabla
\psi\|_q.$$ We have used Poincar\'e Inequality for the last inequality.  The condition on
$q$ is required for $|x|^{-p(\frac n{q'}+1)}$ to be integrable at
infinity.
\end{proof}

This discussion extends to local Hardy spaces, that we define now.
We first
define the truncated version of the maximal function, namely
\begin{equation}\label{maximal-tr}
\M_{\varphi}^{(1)}f(x):=\sup_{0<t<1}\left|\left(f\ast\varphi_{t}\right)(x)\right|.
\end{equation}
A tempered distribution $f$
is said to belong to the space $\h^{p}(\bR^{n})$ if
\begin{equation}\label{def-h}
\|f\|_{\h^{p}(\bR^{n})}:=\left
(\int_{\bR^n}\M_{\varphi}^{(1)}f(x)^{p}dx\right)^{\frac 1p}<\infty.
\end{equation}
The atomic decomposition holds for local Hardy spaces, with only atoms associated to balls of radius less than $1$ satisfying the moment condition, see \cite{G}. The previous lemma is valid in the context of $\h^{p}(\bR^{n})$, with $\M^{(1)}$ in place of $\M$.

\section{Hardy spaces of differential forms}\label{difform}
Let us first fix notations and recall standard properties. Let
$\bR^{n}$ be the Euclidean space equipped with its standard
orthonormal basis $\mathcal B=\left\{e_{1},\ldots,e_{n}\right\}$,
and let $\mathcal B^{\ast}=\left\{e^{1},\ldots,e^{n}\right\}$ be
its dual basis.

For  $\ell\in\left\{1,\ldots,n\right\}$, denote by
$\Lambda^{\ell}$ the space  of  $\ell$-linear alternating forms,
 which consists in linear combinations of exterior
products
\begin{equation}
e^{I}=e^{i_{1}}\wedge\ldots\wedge e^{i_{\ell}},
\end{equation}
where $I=(i_{1},\ldots,i_{\ell})$ is any $\ell$-tuple. The standard basis of $\Lambda^{\ell}$ is $\left\{e^{I}\right\}$ where $I$ is an ordered $\ell$-tuple, $1\leq i_{1}<\ldots<i_{\ell}\leq n$. For $\alpha=\sum_{I}\alpha_{I}e^{I}$ and $\beta=\sum_{I}\beta_{I}e^{I}$ in $\Lambda^{\ell}$, we define the inner product of $\alpha$ and $\beta$ as follows
    \begin{equation}
    \left\langle \alpha,\beta\right\rangle:=\sum\alpha_{I}\beta_{I},
    \end{equation}
    where the summation is taken over all ordered $\ell$-tuples.

    The Hodge operator is the linear operator  $\ast:\Lambda^{\ell}\rightarrow\Lambda^{n-\ell}$ defined by
    \begin{equation}
    \alpha\wedge\ast\beta=\left\langle \alpha,\beta\right\rangle e^{1}\wedge\ldots\wedge e^{n}
    \end{equation}
    for all $\alpha,\beta\in\Lambda^{\ell}$.

    An $\ell$-form on $\bR^{n}$ is defined as a function $u:\bR^{n}\rightarrow\Lambda^{\ell}$
    which may be written as
    \begin{equation}
    u=\sum_{I}u_{I}e^{I},
    \end{equation}
    where the $u_{I}$'s are (real-valued) functions on $\bR^{n}$ and
    all the $I$'s are of length $\ell$.
\begin{defn}
Let $\Omega\subset\bR^{n}$ be an open set, $\ell$ a positive integer as above and
     $\mathcal E(\Omega)$ a normed space of functions $f:\Omega\rightarrow\bR$ equipped with
     the norm $\left\|f\right\|_{\mathcal E(\Omega)}$. We say that an $\ell$-form
     $\omega=\sum_{I}\omega_{I}e^{I}$ belongs to $\mathcal E(\Omega,\Lambda^{\ell})$
     if $\omega_{I}\in\mathcal E(\Omega)$ for all ordered $\ell$-tuples $I$, and we
     pose
\begin{equation}
\left\|\omega\right\|_{\mathcal E(\Omega,\Lambda^{\ell})}:=\sum_{I}\left\|\omega_{I}\right\|_{\mathcal E(\Omega)}.
\end{equation}
\end{defn}
    Let $d:\mathcal D'(\Omega,\Lambda^{\ell-1})\rightarrow\mathcal D'(\Omega,\Lambda^{\ell})$
     denote the exterior derivative operator given by
    \begin{equation}
    d\omega=\sum_{k,I}\partial_{k}\omega_{I}e^{k}\wedge e^{I}
    \end{equation}
where $\partial_{k}\omega_{I}$ is the partial derivative with
respect to the $k$-th variable. The Hodge operator
$\delta:\Lambda^{\ell}\rightarrow\Lambda^{\ell-1}$ defined by
$\delta=(-1)^{n(n-\ell)}\ast d \ast$ is the formal adjoint of $d$
in the sense that if $\alpha\in\mathcal
C^{\infty}(\Omega,\Lambda^{\ell})$ and $\beta\in\mathcal
C^{\infty}(\Omega,\Lambda^{\ell+1})$, then
    \begin{equation}
    \int_{\Omega}\left\langle \alpha,\delta\beta\right\rangle =-\int_{\Omega}\left\langle d\alpha,\beta\right\rangle ,
    \end{equation}
    provided that one of these forms has compact support. We also define the Laplacian
    \begin{equation}
    \Delta_\ell=d\delta+\delta d:\mathcal D'(\bR^{n},\Lambda^{\ell})\rightarrow\mathcal D'(\bR^{n},\Lambda^{\ell})
    \end{equation}
    and a simple calculation shows that for $\omega=\sum_{I}\omega_{I}e^{I}\in W^{2,p}(\bR^{n},\Lambda^{\ell})$ with $1\leq p\leq\infty$,
    \begin{equation}
    \Delta_\ell\omega=\sum_{I}\Delta\omega_{I}e^{I},
    \end{equation}
    where  $\Delta\omega_{I}$ is the usual Laplacian on functions.

For $f=\sum_{I}f_{I}e^{I}\in \mathcal D'(\bR^{n},\Lambda^{\ell})$, we put
 \begin{equation}
 \partial_{j}f:=\sum_{I}\partial_{j}f_{I}e^{I}.
 \end{equation}

\begin{defn}
Let $\ell\in\left\{1,\ldots,n-1\right\}$, and $\frac{n}{n+1}<p\leq
1$. The Hardy space of closed $\ell$-forms is defined as
\begin{equation}
\H^{p}_{d}\left(\bR^{n},\Lambda^{\ell}\right):=\left\{f\in\H^{p}\left(\bR^{n},\Lambda^{\ell}\right):
\quad df=0 \right\}
\end{equation}
endowed with the norm of $\H^{p}(\bR^{n},\Lambda^{\ell})$.
\end{defn}

Recall that all closed $\ell$-forms are exact, that is, there
exists  some $ g\in\mathcal D'(\bR^{n},\,\Lambda^{\ell-1})$ such
that $f=dg$.

We will need the analogue of  the previous scalar
characterizations of Hardy spaces.

For $1\leq q\leq\infty$, we first define, for  $ f\in
L^{1}_{loc}(\bR^{n},\,\Lambda^{\ell})$, the grand maximal function
$\vec{\mathcal M}_{q} f$ as follows.
\begin{equation}
\vec{\mathcal M}_{q}
f(x):=\sup_{\Phi\in\vec{\F}^{q}_{x}}|\int_{\bR^{n}}
f\wedge\Phi|,\label{vect-max}
\end{equation}
where $\vec{\F}^{q}_{x}$ denote the set of all $\Phi\in
W^{1,q}(\bR^{n},\,\Lambda^{n-\ell})$ for which there exists $r>0$
such that $\Phi$ is supported in the ball $B_{(x,r)}$ and satisfies
 $$(*)\quad\left\|\Phi\right\|_{L^{q}(\bR^{n},\,\Lambda^{n-\ell})}
    +r\left\|\nabla\Phi\right\|_{L^{q}(\bR^{n},\,\Lambda^{n-\ell})}\leq\left|B_{(x,r)}\right|^{-\frac{1}{q'}}.$$

The next lemma is a direct consequence of Lemma
\ref{classical-result1} and the fact that for $
f=\sum_{I}f_{I}e^{I}\in\H^{p}(\bR^{n},\Lambda^{\ell})$ and  any
positive integer $k$,
\begin{equation}
    \vec{\mathcal M}_{q} f\leq\sum_{I}\mathcal M_{q} f_{I}.
\end{equation}

\begin{lem}\label{control}
Let  $\frac{n}{n+1}<p\leq 1$ and $\frac 1q <\frac {n+1}n-\frac
1p$. There exists a constant $C$ such that, for all $f\in
L^{1}_{loc}(\bR^{n},\Lambda^{\ell})$,
\begin{equation}
\left\|\vec{\mathcal M}_{q} f\right\|_{L^{p}(\bR^{n})}\leq C
\left\|f\right\|_{\H^{p}(\bR^{n},\Lambda^{\ell})}\label{norm-control2a}.
\end{equation}
\end{lem}

We need a weaker  version of this grand maximal function,
denoted by $\vec{\mathcal M}_{q,d} f$, which is adapted to
Hardy spaces of closed forms. We define
\begin{equation}
\vec{\mathcal M}_{q,d} f(x):=\sup_{\Phi\in
\vec{\F}^{q}_{x,d}}|\int_{\bR^{n}} f\wedge\Phi|,
\end{equation}
where $\vec{\F}^{q}_{x,d}$ denote the set of $ \Phi\in
L^\infty(\bR^{n},\Lambda^{n-\ell})$ supported in some ball $B(x,r)$ satisfying
$$(**)\quad\left\|\Phi\right\|_{L^{q}(\bR^{n},\Lambda^{n-\ell})}
+r\left\|d\Phi\right\|_{L^{q}(\bR^{n},\Lambda^{n-\ell+1})}\leq\left|B_{(x,r)}\right|^{-\frac{1}{q'}}.$$

 \begin{lem}\label{maximal-d}
 Let $q>n$ and $1\leq \ell \leq n-1$. For all $f\in L^{1}_{loc} (\bR^{n},\Lambda^{\ell})$, the following inequality holds
 \begin{equation}
 \vec{\mathcal M}_{q}f\leq\vec{\mathcal M}_{q,d}f.\label{pointwise2a}
 \end{equation}
 Moreover, if $f$ is a closed form, then
  \begin{equation}
\vec{\mathcal M}_{q,d}f \leq C\vec{\mathcal M}_{q}f.
\label{pointwise2b}
 \end{equation}
for some uniform constant $C$.
 \end{lem}

 \begin{proof}

 Let $\Phi=\sum_{I}\Phi_{I}e^{I}\in \mathcal C^{\infty} (\bR^{n},\Lambda^{n-\ell})$. It follows from the fact that
$\displaystyle d\Phi=\sum_{I,j}\partial_{j}\Phi_{I}e^{j}\wedge e^{I}$, that
 \begin{equation} \left\|d\Phi\right\|_{L^{q}(\bR^{n},\Lambda^{n-\ell+1})}\leq\sum_{I,j}\left\|\partial_{j}\Phi_{I}\right\|_{L^{q}(\bR^{n})}\leq\left\|\nabla\Phi\right\|_{L^{q}(\bR^{n},\Lambda^{n-\ell})}.
 \end{equation}

 Thus, for all $x\in\bR^{n}$, we have $\vec{\F}^{q}_{x}\subset\vec{\F}^{q}_{x,d}$ so that
 (\ref{pointwise2a}) follows from the definition of the maximal functions $\vec{\mathcal M}_{q}f$
 and $\vec{\mathcal M}_{q,d}f$.

 Assume  now that  $f$ is a locally integrable closed form. Remark first that, for
 $\phi$ and $\psi$  bounded compactly supported  such that $d\psi=d\phi$, we have

 \begin{equation}\label{egalite}
 \int f\wedge \phi =\int f\wedge \psi.
 \end{equation}
Indeed, we can assume by regularization that $f$ is a smooth
function on some open set containing the supports of $\phi$ and
$\psi$. Moreover, $f$ may be written as $dg$, with $g$ a smooth
function on this open set. So the equality follows from
integration by parts.

Now, let $x\in\bR^{n}$ and
$\Phi=\sum_{I}\Phi_{I}e^{I}\in\vec{\F}^{q}_{x,d}$ supported in $
 B_{(x,r)}$. We put $\varphi(y)=r^{n}\Phi(x+ry)$
for all $y\in\bR^{n}$. Then $\varphi$ is supported in $\bB$ and
 \begin{equation}
 d\varphi(y)=r^{n+1}\sum_{I,j}(\partial_{j}\Phi_{I})(x+ry)e^{j}\wedge e^{I}=r^{n+1}d\Phi(x+ry).
\end{equation}
  So, we obtain

  \begin{equation}
  \left\|d\varphi\right\|_{L^{q}(\bB,\Lambda^{n-\ell+1})}
 = r\left\|d\Phi\right\|_{L^{q}(\bR^{n},\Lambda^{n-\ell+1})}\leq\left|B_{(x,r)}\right|^{-\frac{1}{q'}},
  \end{equation}
  according to the definition of $\Phi\in\vec{\F}^{q}_{x,d}$. To
  conclude for the lemma, it is sufficient to find $\psi$ in $
  W^{1,q}(\bR^n,\Lambda^{n-\ell})$ supported in $\bB$ and
such that $d\psi=d\varphi$ with
\begin{equation}\label{sol-d}
\|\psi\|_{W^{1,q}(\bR^n,\Lambda^{n-\ell})}
  \leq C\left\|d\varphi\right\|_{L^{q}(\bB,\Lambda^{n-\ell+1})}.
\end{equation}
Indeed, if we  let $\Psi(y)=\psi_{r}(y-x)$,  then
$C^{-1}{\Phi}\in \vec{\F}^{q}_{x}$, and $d\Psi=d\Phi$, so
that $\int f\wedge \Phi = \int f\wedge \Psi$.

So we conclude easily from the following lemma.

\begin{lem}\label{le-d}
Let $1<q<\infty$  and $1\leq\ell\leq n-1$. Let $\bB$ be the unit
ball. Let $\varphi\in L^\infty(\bR^n, \Lambda^{\ell})$ compactly
supported in $\bB$ such that $d\varphi$ is in $L^q(\bR^n,
\Lambda^{\ell+1})$. Then there exists $\psi\in
W^{1,q}(\bR^n,\Lambda^{\ell})$ vanishing outside $\bB$, such that
$d\psi=d\varphi$. Moreover, we can choose $\psi$ such that
$$\|\psi\|_{W^{1,q}(\bR^n,\Lambda^{\ell})}
    \leq C\left\|d\varphi\right\|_{L^{q}(\bR^n,\Lambda^{\ell+1})}$$
for some uniform constant $C$.
\end{lem}

\begin{proof} The existence of a form $\psi\in W^{1,q}_0(\bB,\Lambda^{\ell})$ such that $d\psi=d\varphi$ is given by Theorem 3.3.3 of \cite{Sc}.
Moreover, one has the  inequality
$$\|\psi\|_{W^{1,q}(\bB,\Lambda^{\ell})}
    \leq C\left\|d\varphi\right\|_{L^{q}(\bB,\Lambda^{\ell+1})}.$$
Then $\psi$ extends into a form of $W^{1,q}(\bR^n,\Lambda^{\ell})$
when given the value $0$ outside the unit ball. We still note
$\psi $ the form on $\bR^n$, which is supported by $\bB$.
\end{proof}
This allows to conclude for the proof of Lemma \ref{maximal-d}.
\end{proof}

\section{Wedge products}
 We are interested in estimates of wedge products of two differential forms
 of degree $\ell$ and $n-\ell$ respectively, with $1\leq \ell\leq n-1$.  Recall that,
for $f=\sum_{I}f_{I}e^{I}\in \mathcal{C}(\mathbb
R^{n},\Lambda^{\ell})$ and $g=\sum_{J}g_{J}e^{J} \in
\mathcal{C}(\mathbb R^{n},\Lambda^{\mathfrak n-\ell})$,  with $I$
varying among all ordered $\ell$-tuples $1\leq
i_{1}<\ldots<i_{\ell}\leq n$ and $J$ among all ordered
$n-\ell$-tuples, we put
\begin{equation}
f\wedge g=\sum_{I,J}\left(f_{I}\cdot g_{J}\right)e^{I}\wedge
e^{J}.
\end{equation}
The $n$-form $f\wedge g$ identifies with a function via the Hodge
operator. It is clear that the wedge product can also be  defined
as soon as products are. In particular, it is the case when $f\in
L^p(\mathbb R^{n},\Lambda^{\ell})$ and $g\in L^q(\mathbb
R^{n},\Lambda^{n-\ell})$, with $\frac 1p+\frac 1q\leq 1$. Using the
results of \cite{BIJZ} and \cite {BF}, it is also the case when
one of the two forms belongs to the Hardy space $\H^p(\bR^n,
\Lambda^\ell)$ while the other one is in the dual space. Moreover,
it is proved that
\begin{equation}
f\wedge g\in
L^{1}(\bR^{n},\Lambda^{n})+\H^{\Phi}_\omega(\bR^{n},\Lambda^{n})
\end{equation}
if $f\in \H^{1}(\mathbb R^{n},\Lambda^{\ell})$ and $g \in
\bmo(\mathbb R^{n},\Lambda^{n-\ell})$, while
\begin{equation}
 f\wedge g\in L^{1}(\bR^{n},\Lambda^{n})+\H^{p}(\bR^{n},\Lambda^{n})
\end{equation}
if $p<1$,   $f\in \H^{p}(\mathbb R^{n},\Lambda^{\ell})$ and $g \in
\Lambda_{n(\frac{1}{p}-1)}(\mathbb R^{n},\Lambda^{n-\ell})$. Here
$\H^{\Phi}_\omega(\bR^{n},\Lambda^{n})$ is the Hardy Orlicz space
associated to the function $\Phi(t)=\frac{t}{\log(e+t)}$ and $\omega(x)=(\log (e+|x|))^{-1}$.

\medskip

We are now interested in improving these estimates when $f$ and
$g$ are closed. The $\div$-$\curl$ lemma can be generalized to
closed forms: this has already been observed by Lou and Mc Intosh
in \cite{LM1} when $\frac 1p+\frac 1q=1$. In general, we can state
the following.

\begin{thm}\label{main1}
Let $\frac{n}{n+1}<p\leq 1$ and $1\leq\ell\leq n-1$. Let $1<q\leq
\infty$ be such that $\frac 1r:=\frac 1p+\frac 1q\leq
\frac{n+1}{n}$. Then, if $f\in \H^{p}_{d}(\mathbb
R^{n},\Lambda^{\ell})\cap L^{q'}(\mathbb R^{n},\Lambda^{\ell})$
and $g \in L^{q}(\mathbb R^{n},\Lambda^{n-\ell})$ is such that
$dg=0$, then $ f\wedge g\in\H^{r}_{d}(\bR^{n},\Lambda^{n})$.
Moreover, there exists a constant $C$ not depending on $f$ and
$g$, such that
\begin{equation}
\left\| g\wedge f\right\|_{\H^{r}_{d}(\bR^{n},\Lambda^{n})}\leq C
\left\|g\right\|_{L^{q}(\bR^{n},\Lambda^{n-\ell})}\left\|f\right\|_{\H^{p}_{d}(\bR^{n},\Lambda^{\ell})}.
\label{controle1}
\end{equation}
\end{thm}
\begin{proof}
Remark that the forms $f\in \H^{p}_{d}(\mathbb
R^{n},\Lambda^{\ell})\cap L^{q'}(\mathbb R^{n},\Lambda^{\ell})$
are dense in $\H^{p}_{d}(\mathbb R^{n},\Lambda^{\ell})$: just take
$f *P_\varepsilon$, where $P_t$ is the Poisson kernel, to approach
$f$. Remark also that the assumptions can be made symmetric: just
replace $\ell$ by $n-\ell$.

To adapt the proof given in \cite{CLMS} the main point is given in
the next lemma, which has its own interest.

\begin{lem}\label{lemmedeMeyer}
Let $\frac{n}{n+1}<p\leq 1$ and $1\leq\ell\leq n-1$.  Then, for
$f\in \H^{p}_{d}(\mathbb R^{n},\Lambda^{\ell})$, there exists
$h\in L^{p^*}_{d}(\mathbb R^{n},\Lambda^{\ell-1})$ such that
$dh=f$ and $\delta h=0$, with $\frac 1{p^*}=\frac 1p-\frac 1n$.
Moreover $h$ is unique up to the addition of a constant form, and
there exists some uniform constant $C$ such that
\begin{equation}\label{meyer}
\|\Msob(f)\|_p\leq C\|f\|_{\mathcal H^p},
\end{equation}
 with $$\Msob (f)(x)=\sup_{t>0}\frac
1{t|B(x,t)|}\int_{B(x,t)}|h(y)-h_{B(x,t)}|dy.$$
\end{lem}
Recall that $h_{B(x,t)}$ is the mean of $h$ over the ball $B(x,t)$,
which is well defined since $h$ is in $L^{p^*}$ and $p^*>1$.

Remark that $\Msob (f)$ is independent on the choice of $h$ since $h$ is unique up to the addition of a constant form.
\begin{proof}
The case $\ell=1$ is Lemma II.2 of \cite{CLMS}. So it is
sufficient to consider $\ell>1$. Let us first remark that the
uniqueness is direct, since in $L^{p^*}(\bR^n, \Lambda^{\ell-1}) $
only constants are $d$-closed and $\delta$-closed. Assume that $h$
is a solution. Then all the derivatives $\partial_j h_I$ are in the
Hardy space $\H^p(\bR^n)$. Indeed, we use the fact that, by
definition of $\Delta_{\ell-1}$, we have the identities
$$\Delta_{\ell-1}h=\delta f, \qquad \qquad \partial_j h=
(\partial_j (-\Delta_{\ell-1})^{-1/2})(
(-\Delta_{\ell-1})^{-1/2}\delta)f.$$ But both operators arising in
the last expression, that is $\partial_j
(-\Delta_{\ell-1})^{-1/2}$ and $(-\Delta_{\ell-1})^{-1/2}\delta$,
are linear combinations of Riesz transforms and preserve Hardy
spaces. Indeed, since $\Delta_{\ell-1}$ is given by the Laplacian,
coefficients by coefficients, the same is valid for all its
powers. Furthermore, we have
$$\Vert \nabla h\Vert_{\mathcal H^p}\le C\Vert f\Vert_{\mathcal H^p}.$$

Conversely, given $f$, we can use these formulas to fix
the values of $\partial_j h_I$ in the Hardy space $\H^p(\bR^n)$.
Using this lemma for one-forms, we know the existence of $h_I\in
L^{p^*}(\bR^n)$ having these functions as derivatives. It is
elementary to see that the form $h=\sum h_I e^I$ is such that $dh=f$ and
$\delta h=0$, using commutation properties of the operators.
Finally, we write \eqref{meyer} for each $f_j=\partial_jh_I$ to obtain the
inequality for $f$.
\end{proof}
It is elementary to adapt the rest of the proof of Theorem II.3 in
\cite{CLMS}, once the lemma has been settled, and we leave it to
the reader. So this gives the  proof of Theorem \ref{main1}.
Remark that  $f\wedge g$ can be defined in the distribution sense
without the additional assumption that $g\in L^{q'}(\mathbb
R^{n},\Lambda^{\ell})$: we pose $f\wedge g = d(h\wedge g)$, with
$h$ given by Lemma \ref{lemmedeMeyer}. Indeed, $h\wedge g$ is in
the Lebesgue space $L^s$, with $s>1$ given by $\frac 1s = \frac
1r-\frac 1n$. So its exterior derivative is well defined as a
distribution.
\end{proof}

\medskip

Following the ideas of \cite{ART}, let us sketch another proof for
the endpoint $q=\infty$.

\begin{proof}[Proof for the endpoint]

Let $f\in \H^{p}_{d}(\mathbb R^{n},\Lambda^{\ell})$ and $g \in
L^{\infty}(\mathbb R^{n},\Lambda^{n-\ell})$ such that $dg=0$.  We
want to prove that
$$ \mathcal M \left(f\wedge g\right)(x)\leq C
\left\|g\right\|_{L^{\infty}(\bR^{n},\Lambda^{n-\ell})}\vec{\mathcal
M}_{\infty,d} f(x),
$$
from which we conclude directly by using Lemma \ref{control} and
Lemma \ref{maximal-d}. By linearity we can assume that
$\|g\|_\infty=1$. In order to estimate $\M (f\wedge g)(x)$, we
have to consider
$$\int \varphi_{x,t} f\wedge g, \qquad \qquad
\varphi_{x,t}(y)=t^{-n}\varphi ((x-y)/t).$$ Here $\varphi$ is
chosen smooth and supported in the unit ball as in
\eqref{maximal}. It is sufficient to prove the existence of some
uniform constant $c$ such that $c\varphi_{x,t}\, g$ belongs
to $\vec{\F}^{\infty}_{x,d}$. This follows from the inequality
$$\|\varphi_{x,t}\|_\infty + t\|d\varphi_{x,t}\|_\infty\leq
|B_{(x,t)}|^{-1}.$$ Indeed, since $g$ is closed, we have the
equality $d(\varphi_{x,t}\, g)=d\varphi_{x,t}\wedge g$ and the
uniform norm of a wedge product is bounded by the product of
norms.

This finishes the proof.
\end{proof}
\section{$\BMO$ estimates}

Let us first recall some facts on $BMO(\bR^{n})$ and weighted Hardy-Orlicz
spaces.

Given a continuous function
$\mathcal P :\left[0,\infty\right)\rightarrow\left[0,\infty\right)$ increasing
 from 0 to $\infty$ (but not necessarily convex, $\mathcal P$ is called the Orlicz function), and given a positive measurable function $\omega$,
 the weighted Orlicz space $L^{\mathcal P}_\omega(\bR^{n})$ consists in the set of functions $f$ such that
\begin{equation}
\left\|f\right\|_{L^{\mathcal P}_\omega(\bR^{n})}:=\inf\left\{k>0:\int_{\bR^{n}}\mathcal P(k^{-1}\left|f\right|)\omega(x)dx\leq 1\right\}
\end{equation}
is finite. The Hardy Orlicz space $\H^{\mathcal P}_\omega(\bR^{n})$ (resp. the local Hardy Orlicz space $\h^{\mathcal P}_\omega(\bR^{n})$ is the space of tempered distributions $f$ such that $\mathcal Mf$ (resp. $\mathcal M^{(1)}f$) belongs to $L^{\mathcal P}_\omega(\bR^{n})$.
 We will consider here the Orlicz space
 associated to the function $\Phi(t)=\frac{t}{\log(e+t)}$
and the weight $(\log (e+|x|))^{-1}$, as mentioned in the introduction. The space $L^{\Phi}_\omega(\bR^{n})$ is not a normed space. Let us state the following properties of the function  $\Phi$ and $L^{\Phi}_\omega$.
\begin{eqnarray}
 \quad \Phi(s+t) &\leq & \Phi(s)+\Phi(t), \label{sub-add} \qquad\qquad \qquad\qquad  \mbox{\rm for all }\ s,t>0.\\
 \quad \Phi(st) &\geq & \Phi(s)\Phi(t), \label{upper-mult} \qquad\qquad \qquad\qquad  \mbox{\rm for all }\ s,t>0.\\
 \Phi(st)&\leq &  s + e^t-1 \qquad\qquad\qquad\qquad \mbox {\rm for all }\ s,t>0.\label{conj}\\
  \Phi(st)&\leq & \log (e+d) \;( s +\frac {1}d (e^t-1)) \qquad\qquad \mbox {\rm for all }\ s,t, d>0.\label{conjugate}
\end{eqnarray}
The two first inequalities are elementary. The third one is given in \cite{BIJZ}. For $d>0$, we write that $\Phi(d)\Phi(st)\leq \Phi((sd)t)$, and use the previous one with $sd$ in place of $s$. 

Next, by using the fact that $\Phi(4)>2$ and \eqref{upper-mult} we obtain the inequality
$$\Phi\left(\frac{s+t}4\right)\leq \frac {\Phi(s)+\Phi(t)}2,$$ from which we conclude that
\begin{equation}\label{additivity}
\left\|f+g\right\|_{L^{\Phi}_\omega(\bR^{n})}\leq 4\left\|f\right\|_{L^{\Phi}_\omega(\bR^{n})}+4\left\|g\right\|_{L^{\Phi}_\omega(\bR^{n})}.
\end{equation}
 We will also need the fact that products of integrable functions
with functions in the exponential class are in $L^\Phi$. More
precisely, we will use the following lemma, for $B$ a ball with radius $1$. It is a direct consequence of \eqref{conjugate}.
\begin{lem}\label{holder} For $c>0$ given, there exists $C$ such that, for $d>0$,
\begin{equation}
\int_B \Phi(|fg|)\frac {dx}{\log (e+d)}\leq C  \int_B |f| dx +\frac {C}d \int_B (e^{c|g|}-1)dx.
\end{equation}
\end{lem}

\medskip
Our main theorem is  the following.
\begin{thm}\label{bijz}
Let $f\in \H^{1}_{d}(\bR^{n},\Lambda^{\ell})$ and $g\in \bmo
(\bR^{n},\Lambda^{n-\ell})$ such that $dg=0$, then the product
$f\wedge g$ is in $\H^{\Phi}_\omega(\bR^{n},\Lambda^{n})$. Moreover,
there exists a uniform constant $C$ such that
\begin{equation}
\left\|f\wedge g\right\|_{\H^{\Phi}_\omega(\bR^{n},\Lambda^{n})}
\leq C\left\|f\right\|_{\H^{1}(\bR^{n},\Lambda^{\ell})}\left\|g\right\|_{\bmo(\bR^{n},\Lambda^{n-\ell})}.
\end{equation}

\end{thm}

\begin{proof}
Recall that the wedge product $f\wedge g$ is well defined in the
distribution sense for $f\in\H^{1}_{d}(\bR^{n},\Lambda^{\ell})$
and $g\in\bmo(\bR^{n},\Lambda^{n-\ell})$. It is sufficient to have
an a priori estimate for $g$ bounded, which we assume from now on.
We also assume that
$\left\|g\right\|_{\bmo(\bR^{n},\Lambda^{n-\ell})}=1$ and $\left\|f\right\|_{\H^{1}(\bR^{n},\Lambda^{\ell})}=1$.  Then, for
$x\in\bR^{n}$ and $\varphi\in \F^{\infty}_{x}$ supported in
$B(x,r)$, we have
\begin{equation}\label{maximal-function}
\left|\int\left(f\wedge g\right)\varphi\right|\leq\left|\int f\wedge\left(g-g_{B(x,r)}\right)\varphi\right|+\left|\int f\wedge(g_{B(x,r)}\varphi)\right|,
\end{equation}
where
\begin{equation}
g_{B(x,r)}=\sum_{I}(g_{I})_{B(x,r)}e^{I}\text { if }
g=\sum_{I}g_{I}e^{I}.
\end{equation}
Let us first evaluate the second term of the sum
(\ref{maximal-function}). We have
\begin{eqnarray}
\left|\int
f\wedge(g_{B(x,r)}\varphi)\right|&=&\left|\sum_{I}(g_{I})_{B(x,r)}\int
f\wedge\varphi e^{I}\right|\nonumber\\&\leq&\left(1+\mathfrak
M^{(1)} g(x)\right)\vec{\mathcal
M}_{\infty}(f)(x)\label{second-part}
\end{eqnarray}
where $\mathfrak M^{(1)} g(x)=\sum_{I}\mathfrak M^{(1)}
g_{I}(x)e^{I}$, and for $h\in L^{1}_{loc}(\bR^{n})$ a scalar
valued function,
\begin{equation}
\mathfrak M^{(1)}h(x)=\sup\left\{\frac{1}{\left|B\right|}\int_{B}\left|h(y)\right|dy,x\in B \text { and } \left|B\right|<1\right\}.
\end{equation}
Indeed, when $r\geq 1$, the mean of $g$ is, by definition of
$\bmo$, bounded by
$\left\|g\right\|_{\bmo(\bR^{n},\Lambda^{n-\ell})}=1$, while, for
$r<1$, it is bounded by the maximal function related to small
balls. For the first term, we proceed in the same way as for
bounded $g$. By John-Nirenberg Inequality, the form
$\left(g-g_{B(x,r)}\right)\varphi$ satisfies Condition $(**)$ for
all $q>1$ up to some uniform constant $C$. So, according to Lemma
\ref{maximal-d}, for $q$ large enough,
\begin{equation}\label{first-part}
\left|\int f\wedge\left(g-g_{B(x,r)}\right)\varphi\right|\leq
C\vec{\mathcal M}_{q}f(x).
\end{equation}
Taking in  (\ref{second-part}) and (\ref{first-part}) the supremum
over all $\varphi\in \F^{\infty}_{x}$ yields
\begin{equation}\label{decomposition2}
\mathcal M_{\infty}(f\wedge g)\leq C\left(\vec{\mathcal
M}_{q}(f)+\mathfrak M^{(1)} (g)\vec{\mathcal
M}_{\infty}(f)\right).
\end{equation}
 The first term is in $L^{1}(\bR^{n})$ under the assumption that $f\in\H^{1}_{d}(\bR^{n},\Lambda^{\ell})$
  according to Lemma \ref{control}. It remains to prove that the second term is in $L^{\Phi}_\omega(\bR^{n})$
  to conclude the proof of our theorem. Roughly speaking, this
  is the consequence of the fact that the product of an integrable function with a function in the exponential class
  is in the Orlicz space $L^{\Phi}$. Indeed, we recall that, by John-Nirenberg's Inequality, there exists $c>$ such that, for each ball $B$,
 $$
  \int_{B}e^{c\left|g-g_{B}\right|}dx\leq C.$$
  The following lemma allows to have the same kind of estimate for $\mathfrak M^{(1)}(g)$ in place of $g$.
 \begin{lem}
 Let $c>0$ fixed. Then there exists some uniform constant $C$ such that
 \begin{equation}
 \int_{B}e^{\frac c2\left|\mathfrak M^{(1)}(g)\right|}\leq C  \int_{3B}e^{c|g|}dx,
 \end{equation}
 for every ball $B$ with radius  $1$. The ball $3B$ is the ball with same center and radius $3$.
 \end{lem}

 \begin{proof}
  We have
 \begin{equation}
 \int_{B}e^{\frac c2\left|\mathfrak M^{(1)}(g)(x)\right|}dx\leq e^c|B|+\int^{\infty}_{1}e^{\frac {cs}2}\left|\left\{x\in B: \mathfrak M^{(1)}(g)>s\right\}\right|ds.
 \end{equation}
 Let $g_{1}=g\chi_{3B}$, where $3B$ is the ball having same center as $B$, but with radius $3$ times the one of $B$ and $\chi_{3B}$ is the characteristic function of $3B$. It is easy to see that for all $s>0$, we have
$$
 \left\{x\in B : \mathfrak M^{(1)}g(x)>s \right\}\subset\left\{x\in\bR^{n}: \mathfrak M g_{1}(x)>s\right\}.
 $$
Thus from the above inclusion and the weak type $(1,1)$ boundedness of the Hardy-littlewood maximal function, we have
 $$
 \left|\left\{x\in B: \mathfrak M^{(1)}(g)>s\right\}\right|\leq\left|\left\{x\in\bR^{n}: \mathfrak M g_{1}(x)>s\right\}\right|\leq\frac{C}{s}\int_{\left\{\left|g_{1}\right|>\frac{s}{2}\right\}}\left|g_{1}(x)\right|dx
$$
 with  $C$ independent of $B$ and $g$, so that the integral in the second member becomes
 \begin{eqnarray*}
\int^{\infty}_{1}e^{\frac {cs}2}\left|\left\{x\in B: \mathfrak M^{(1)}(g)>s\right\}\right|ds
&\leq& C\int^{\infty}_{1}\frac{e^{\frac {cs}2}}{s}\left(\int_{\left\{\left|g\right|>\frac{s}{2}\right\}\cap 3B}\left|g(x)\right|dx\right)ds\\
&\leq&C' \int_{3B}e^{c\left|g(x)\right|}dx
 \end{eqnarray*}
 by using Fubini's Theorem.  We conclude for the lemma.
 \end{proof}
 Let us come back to the proof of Theorem \ref{bijz}. We  write $\bR^n$ as the almost disjoint union of balls $B_j=:\bB+j$, with $j\in \bZ^n$, with $\bB$ the unit ball centered at $0$. We  make use of Lemma \ref{holder} on each of these balls, with $d:=d_j:=(1+|j|)^{-N}$, with $N$ large enough so that $\sum d_j^{-1}<\infty$ while $\omega(x)\simeq (\log d_j)^{-1}$ for $x\in B_j$. We recall that by assumption $|g_{3B_j}|\leq C$, since
 $\left\|g\right\|_{\bmo(\bR^{n},\Lambda^{n-\ell})}=1$. So

 \begin{equation}\label{from-bmo}
    \int_{3B_j} e^{c\left|g(x)\right|}dx \leq C.
\end{equation}
 We finally have
 \begin{equation}
 \int_{\bR^n}\Phi\left(\mathfrak M^{(1)}(g)(x)
 \vec{\mathcal M}_{\infty}(f)(x)\right)\omega(x) dx \leq C\sum_{j\in \bZ^n} \int_{B_j} |f|dx,
 \end{equation}
 from which we conclude that $\mathfrak M^{(1)}(g)\vec{\mathcal M}_{\infty}(f)$
  has a bounded norm in $L^{\Phi}_\omega(\bR^{n})$ because of the finite overlapping of balls $B_j$.
  By finite additivity (\ref{additivity}) we have
$$
 \left\|\mathcal M(f\wedge g)\right\|_{L^{\Phi}_\omega(\bR^{n})}\leq 4C\left\|\vec{\mathcal M}_{q}f\right\|_{L^{\Phi}_\omega(\bR^{n})}+4C\left\|\mathfrak M^{(1)}(g)\vec{\mathcal M}_{\infty}(f)\right\|_{L^{\Phi}_\omega(\bR^{n})}
 $$
 so that
 \begin{equation}
 \left\|f\wedge g\right\|_{\H^{\Phi}_\omega(\bR^{n},\Lambda^{n})}\leq C
 \end{equation}
 for $\left\|f\right\|_{\H^{1}(\bR^{n},\Lambda^{\ell})}=1$.
 \end{proof}

 We have as well the following theorem.

 \begin{thm}\label{bijz2}
Let $f\in \h^{1}_{d}(\bR^{n},\Lambda^{\ell})$ and $g\in \BMO
(\bR^{n},\Lambda^{n-\ell})$ such that $dg=0$, then the product
$f\wedge g$ is in $\h^{\Phi}_\omega(\bR^{n},\Lambda^{n})$. Moreover,
there exists a uniform constant $C$ such that
\begin{equation}
\left\|f\wedge g\right\|_{\h^{\Phi}_\omega(\bR^{n},\Lambda^{n})}
\leq C\left\|f\right\|_{\h^{1}(\bR^{n},\Lambda^{\ell})}\left\|g\right\|_{\BMO(\bR^{n},\Lambda^{n-\ell})}.
\end{equation}
\end{thm}
The key point is that again we only have to make use of $\mathfrak M^{(1)}g$ and not of $\mathfrak M g$.
The only difference in the proof is the replacement of \eqref{from-bmo} by
\begin{equation}\label{from-BMO}
    \int_{3B_j} e^{c\left|g(x)\right|}dx \leq C(1+|j|)^M
\end{equation}
for some $M>0$: use the well-known fact that $g_{3B_j}\leq C \log(1+|j|)$.
\medskip

 The generalization of Theorem \ref{bijz} to $\H^p_\omega$ for $p<1$ is direct from \eqref{decomposition2}. Then the product
 $f\wedge g$ belongs to $\H^{\Phi_p}_\omega$, with $\Phi_p(t)
 =\left(\frac{t}{\log(e+t)}\right)^p$. Remark that, in this case, the product
 of a function in $\bmo$ and a function in $\H^p$ does not make
 sense as a distribution in general. But we can establish as above
 an a priori estimate, which allows to give a meaning to the wedge
 product of two closed forms.

\end{document}